\documentclass[reqno,11pt]{amsart}
\usepackage{amsmath,amssymb,color}
\newtheorem{theorem}{Theorem}[section]
\newtheorem{proposition}[theorem]{Proposition}
\newtheorem{definition}[theorem]{Definition}
\newtheorem{lemma}[theorem]{Lemma}
\newtheorem{corollary}[theorem]{Corollary}
\newtheorem{Remark}[theorem]{Remark}

\numberwithin{equation}{section}
\numberwithin{theorem}{section}

\newcommand{\bb}[1]{{\mathbb #1}}
\newcommand{\chempote}{e_\eps}
\newcommand{\chempotek}{e_{\eps_k}}
\newcommand{\eps}{\varepsilon}
\newcommand{\epss}{\varepsilon^2}
\newcommand{\epseps}{\eta}
\newcommand{\F} {{F^{**}}}
\newcommand{\Fe} {{F_\eps}}
\newcommand{\Fek} {F_{\eps_k}}
\newcommand{\grad} {\nabla}
\newcommand{\gradflowu} {u}
\newcommand{\gradflowue} {{u_\eps}}
\newcommand{\gradflowuek} {{u_{\eps_k}}}

\newcommand{\gus}{\Sigma_G}
\newcommand{\guse}{\Sigma_G^{\epseps}}
\newcommand{\gusrho}{\Sigma_G^{\rho}}
\newcommand{\hwe}{{\widehat w_\eps}}
\newcommand{\inidatu}{\overline u}
\newcommand{\inidatue}{\overline u_\eps}
\newcommand{\lus}{\Sigma_L}

\newcommand{\mc}[1]{{\mathcal #1}}

\newcommand{\nada}[1]{}
\newcommand{\ovgus}{\overline\Sigma_G}
\newcommand{\R}{\mathbb R}
\newcommand{\some}{w}
\newcommand{\staticu}{\ustatic}
\newcommand{\spazio}{{\mathcal H^{-1}_m}(\toro)}
\newcommand{\toro}{\mathbb T}
\newcommand{\ue}{u_\eps}
\newcommand{\uek}{u_{\eps_k}}
\newcommand{\uestatic}{\ve}
\newcommand{\uekstatic}{\vek}
\newcommand{\ustatic}{v}
\newcommand{\var}{\xi}
\newcommand{\ve}{v_\eps}
\newcommand{\vek}{v_{\eps_k}}
\newcommand{\we}{w_\eps}
\newcommand{\xe}{x_\eps}
\newcommand{\ye}{y_\eps}

\newfont{\indic}{bbmss12}

\definecolor{light}{gray}{.97}

\title{
Convergence of the one-dimensional Cahn-Hilliard equation}

\author[G.\ Bellettini]{Giovanni Bellettini}
\address{Dipartimento di Matematica, Universit\`a di Roma Tor
Vergata, via della Ricerca Scientifica 1, 00133 Roma, Italy.}
\email{belletti@mat.uniroma2.it}
\author[L.\ Bertini]{Lorenzo Bertini}
\address{Dipartimento di Matematica, Universit\`a di Roma La Sapienza,
Piazzale Aldo Moro 2, 00185 Roma, Italy.}
\email{bertini@mat.uniroma1.it}
\author[M.\ Mariani]{Mauro Mariani}
\address{Laboratoire d'Analyse,
Topologie, Probabilit\'es (CNRS UMR 6632), Universit\'e 
Aix-Marseille, Avenue Escadrille Normandie-Niemen, 13397 Marseille
Cedex 20, France.}
\email{mariani@cmi.univ-mrs.fr}
\author[M.\ Novaga]{Matteo Novaga}
\address{Dipartimento di Matematica, Universit\`a di Padova,
via Trieste 63, 35121 Padova, Italy.}
\email{novaga@math.unipd.it}

\date{}
\begin{document}
\begin{abstract}
We consider the Cahn-Hilliard equation in one space dimension with scaling parameter $\eps$, 
i.e. $u_t  = (W'(u)-\epss u_{xx})_{xx}$, where $W$ is a nonconvex potential.
In the limit $\eps\downarrow 0$, under the assumption that the initial data are energetically well-prepared,
we show the convergence to a Stefan problem. The proof is based on variational methods and exploits 
the gradient flow structure of the Cahn-Hilliard equation.
\end{abstract}

\maketitle

\thispagestyle{empty}

\section{Introduction}
In this paper we are interested in the convergence of solutions $\gradflowue=\gradflowue(\cdot,\cdot,\inidatu_\eps)$
to the equation 
\begin{equation}\label{eq:chgf}
\begin{cases}
{\rm u}_t  = \big(W'({\rm u})-
\epss {\rm u}_{xx}\big)_{xx} & \text{in $(0,+\infty) \times \bb T$}
\\
{\rm u} = \inidatu_\eps & \text{on $\{0\}\times \toro$}
\end{cases}
\end{equation}
as $\eps \downarrow 0$, where $\toro:=\bb R/\bb Z$ is 
the one-dimensional torus.
Here  $\eps$ is a spatial scale parameter and
$W$ is a rather general smooth potential.
Our analysis covers, in particular,
the choice of the double-well potential
\begin{equation}\label{pdw}
W(\var) = \frac{(1-\var^2)^2}{4}  \qquad \var \in \R,
\end{equation}
corresponding to the Cahn-Hilliard equation.
We refer for instance to \cite{CH, GL}
for the physical motivations leading
to equation \eqref{eq:chgf}, in relation 
with the theory of phase transitions, 
and to \cite{Sl:91,BaPiXu:95b,Ch:96}
for some mathematical results and connections
with the Stefan problem \cite{Me:92}.

Equation \eqref{eq:chgf} can be seen as the gradient
flow, in the $H^{-1}$-topology, of the Allen-Cahn type functional
\begin{equation}\label{aobwsimagnatroppo}
\Fe(\ustatic)= \int_{\toro}\left(
\epss\frac{\ustatic_x^2}{2} +W(\ustatic)\right) dx, 
\end{equation}
where
the scalar field $\ustatic$ 
represents the local order parameter. The gradient flow structure 
of \eqref{eq:chgf} allows us to look at the convergence 
of the functions $\gradflowue$ 
in a purely variational way, at least under the assumption
of energetically well-prepared initial data.

The main difficulty in studying
the limit of $\gradflowue$ 
is due to the fact that, when the function $W$ is nonconvex,
\eqref{eq:chgf} is forward-backward
parabolic for $\eps=0$. 
Looking at equation \eqref{eq:chgf}, it is rather natural to
expect a limit equation related to the $H^{-1}$-gradient flow of the functional
\begin{equation}\label{aobwsimagnamale}
F(v)=\int_{\toro} 
W(\staticu)~dx.
\end{equation}
However, when $W$ is nonconvex, the functional $F$ is not convex 
and not lower semicontinuous 
with respect to the $H^{-1}$-topology, and the gradient flow dynamics is not well-posed. 
The lower semicontinuous envelope of $F$ is given by 
\begin{equation*}
\F(\staticu)=\int_{\toro} W^{**}(\staticu)~dx,
\end{equation*}
where $W^{**}$ denotes the convex envelope of $W$. 
It is not difficult to prove (see Proposition \ref{l:mm}) 
that $\F$ is the $\Gamma$-limit 
of the functionals $\Fe$ as $\eps\downarrow 0$,
with respect to the $H^{-1}$-topology.

In this paper we prove
that the solutions $\gradflowue$ to \eqref{eq:chgf} converge
to the gradient flow of $\F$, as $\eps\downarrow 0$, under a suitable assumption on the 
initial data $\inidatu_\eps$.
Our main result can be informally stated as follows
(see Theorem \ref{teo:convergence} for the precise statement).
Let $\inidatu$ be such that $\F(\inidatu) < +\infty$, take a sequence
$(\inidatue)$ of initial data satisfying
$\Fe(\inidatue) < +\infty$, converging to $\inidatu$ in $H^{-1}(\toro)$
such that 
$$
\int_\toro \inidatue~dx = \int_\toro \inidatu~dx,  
$$
and
\begin{equation}\label{aia}
\lim_{\eps \downarrow 0} \Fe(\inidatue) = \F(\inidatu).
\end{equation}
Then the solution
$\gradflowue(\cdot,\cdot,\inidatue)$ of \eqref{eq:chgf} converges to 
the $H^{-1}$-gradient flow of $\F$, namely to the solution $\gradflowu$ of
\begin{equation}
\label{eq:hgfbis}
\begin{cases}
\partial_t u= \big(W^{**~\!\prime}(u)\big)_{xx}
& \text{in $(0,+\infty)\times \bb T$}
\\
u= \inidatu & \text{on $\{0\}\times \bb  T$},
\end{cases}
\end{equation}
which, for $W$ nonconvex, is the weak formulation of the Stefan problem \cite{Me:92}.

Some comments concerning  hypothesis
\eqref{aia} are in order, related to the so-called
wrinkling phenomenon.
Given $\inidatu\in H^{-1}(\toro)$, define
\begin{equation}\label{sugus}
\gus := \{\var \in \R : W(\var) > W^{**}(\var)\}, \qquad \lus := 
\{\var \in \R : W''(\var) < 0\},
\end{equation}
and
$$
\gus(\inidatu) := \{x \in \toro : \inidatu(x) \in \gus\}, 
\qquad \lus(\inidatu) := \{x \in \toro: \inidatu(x) \in \lus\}.
$$
We call $\gus(\inidatu)$ the global unstable set of $\inidatu$,
and $\lus(\inidatu)$ the local unstable set of $\inidatu$.
Numerical simulations performed 
in \cite{BeFuGu:05} 
(see also \cite{FiGoPa:98}) 
show a quick formation 
of oscillations
and these microstructures
seem to generically appear only  
in $\lus(\inidatu)$, instead that on the whole of $\gus(\inidatu)$.
In addition, superimposing on $\inidatu$ a microstructure 
in a region $\Sigma \subseteq \gus(\inidatu) \setminus \lus(\inidatu)$
leads to a numerical solution which seems to depend on the choice of $\Sigma$. 
These simulations show an instability of solutions
$\gradflowue(\cdot,\cdot,\inidatu)$ with respect to $\inidatu$.
In particular, if we take 
two sequences $(\widetilde u_\eps)$, $(\widehat u_\eps)$
of initial data both approximating $\inidatu$ and corresponding to 
two different choices of $\Sigma$, 
in general one may expect that
%
$$
\lim_{\eps \downarrow 0}
\gradflowue(\cdot,\cdot, (\widetilde u_\eps))
~\neq~
\lim_{\eps \downarrow 0}
\gradflowue(\cdot,\cdot, (\widehat u_\eps)).
$$
%
Hypothesis \eqref{aia}
can thus be interpreted as an energetically well-prepared assumption
on the initial data $\inidatue$, corresponding to the 
choice of the above mentioned region $\Sigma 
=\gus(\inidatu) \setminus \lus(\inidatu)$. It is worth to remark that, in view of the $\Gamma$-limit $F_\eps \to F^{**}$ stated above, given any $\inidatu \in H^{-1}$, there exists 
a sequence $(\inidatue)$ converging to $\inidatu$ and satisfying \eqref{aia}.

The proof of our main result is entirely variational, and 
it is worthwhile to observe that 
we never use directly equation \eqref{eq:chgf}. 
The main point, indeed, is to derive sufficient information 
on a sequence $(\uestatic)$ of functions (independent of time)
satisfying the uniform bound
\begin{equation}\label{bevans}
\sup_{\eps \in (0,1]}
\left\{ \Fe(\uestatic) + \int_{\toro} 
\Big[~ \Big(W'(\uestatic) - \epss {\uestatic}_{xx}\Big)_x~\Big]^2~dx
\right\}
< +\infty.
\end{equation}
We follow an idea formalized by E. Sandier and S. Serfaty in \cite{Serfaty} (see also \cite{Serfy}), 
where it is shown that the convergence
of the gradient flows of a sequence of functionals $\mathcal F_\eps:H\to [0,+\infty]$, where $H$ is a Hilbert space, 
to the gradient flow of $\mathcal F:=\Gamma-\lim \mathcal F_\eps$
is basically a consequence of the $\Gamma$-convergence of the sequence
of the slopes of the gradients 
$\vert \grad \mathcal F_\eps\vert$ 
of $\mathcal F_\eps$ 
to the 
slope of the gradient
$\vert \grad\mathcal F\vert$ of $\mathcal F$.
More precisely, it suffices to show  
the  $\Gamma$-liminf inequality
\begin{equation}\label{gginf}
\Gamma-\liminf_{\eps\to 0}\ 
\vert\nabla \mathcal F_\eps \vert
\ \geq \ 
\vert\nabla\mathcal F\vert.
\end{equation}
The above inequality, in our setting, is the content of Theorem \ref{p:gammaslope}.
We then obtain the corresponding convergence of the gradient flows of $\Fe$ in Theorem \ref{teo:convergence}.
The main difficulty in the proof 
is contained in Lemma \ref{lemmamatteo},
where a careful analysis of the regions where the functions 
$\uestatic$ oscillate is performed.

We mention that the 
same method proposed in \cite{Serfaty} 
has been successfully applied in \cite{Le1,Le2} to show the convergence, in all space dimensions,
of solutions to the rescaled Cahn-Hilliard equation
\begin{equation*}
\begin{cases}
u_t  = \Delta \big(\eps^{-1}W'(u)-\eps \Delta u\big) 
\\
u(0,\cdot) = \inidatu_\eps,
\end{cases}
\end{equation*}
under suitable simplifying assumptions, in particular related
to the validity of the analog of \eqref{gginf}.

We observe that equation \eqref{eq:chgf}
is not the only way to regularize the ill-posed gradient flow equation of 
the functional \eqref{aobwsimagnamale}: other regularizations have been considered in the literature,
see for instance \cite{Pl:94,DeM:96,EvPo:04,GN:11,SmTe:11}.
In particular, in \cite{DeM:96} it is proposed an implicit variational scheme for the functional \eqref{aobwsimagnamale} 
which converges to \eqref{eq:hgfbis} as the discretization parameter tends to zero.
Due to the high instability of the problem, 
different regularizations could in principle lead to different limiting solutions.

\smallskip 

\noindent {\bf Acknowledgements.}
The authors are grateful to the Centro De Giorgi 
of the Scuola Normale di Pisa for the kind hospitality, and 
to the Mathematisches
Forschungsinstitut Oberwolfach for providing a stimulating research environment. The third author acknowledges the support of the ANR SHEPI grant ANR-2010-BLAN-0108.

\section{Notation}\label{s:1}
Let $\bb T:=\bb R/ \bb Z$ 
be the one-dimensional torus of side length $1$, and $dx$ be the 
Lebesgue measure on $\bb T$. For $m\in \bb R$, let
\begin{equation*}
 \spazio:= \big \{\ustatic
 \in H^{-1}(\mathbb T)\,:\:  \langle \ustatic, 1 \rangle =m\}
\end{equation*}
where $\langle \cdot, \cdot \rangle$ denotes the $H^{-1}(\bb T)$-$H^1(\bb T)$ duality. $\spazio$ is a 
closed affine subspace of $H^{-1}(\toro)$, that  will be 
considered equipped with the induced metric. 
The linear space associated with $\spazio$ 
is the homogeneous negative Sobolev space 
$$
\dot H^{-1}(\bb T) \sim \mathcal H^{-1}_0(\toro). 
$$
In the following, we denote by 
$\|\cdot\|_{-1}$ the Hilbert norm on $\dot H^{-1}(\bb T)$, namely
\begin{equation}
\label{e:Hmeno1}
\|\ustatic\|_{-1}^2:= \|\ustatic\|_{\dot H^{-1}(\bb T)}^2= 
\sup_{\varphi \in H^1(\bb T)}\left\{
2  \langle \ustatic,
\varphi\rangle - \|\varphi_x\|_{L^2(\bb T)}^2\right\},
\end{equation}
and we understand $\|\ustatic\|_{-1}:=+\infty$ if $\ustatic \not \in \dot H^{-1}(\bb T)$.

Throughout the paper, we use the term sequence also to denote
families labeled by the continuous positive parameter $\eps$. 
A subsequence
of $(f_\eps)$ is a sequence $(f_{\eps_h})$
with $\eps_h \downarrow  0$ as $h \to +\infty$.

\subsection{Assumptions on $W$}\label{subsec:thepot}
In the sequel we assume that
$W$ is a function in  $\mathcal C^2(\bb R; [0,+\infty))$ 
satisfying the following properties:
\begin{enumerate}
\item[i)] there exists a constant $C>0$ such that 
\begin{equation}\label{a1}
|W'(\var)| \leq C (1+W(\var)), \qquad \var \in \R,
\end{equation}
and 
$$
\lim_{\vert \var \vert \to +\infty} W(\var) = +\infty;
$$
\item[ii)] $W$ is not affine in any interval of $\R$;
\item[iii)] the global unstable set $\gus$ of $W$, as defined in 
\eqref{sugus}, 
is a bounded open set, consisting
of a finite number of open connected components, denoted by 
$$
\Sigma_1,\dots,\Sigma_\ell.
$$
\end{enumerate}
For the standard double-well potential \eqref{pdw} one has $\ell=1$ and $\Sigma_G=\Sigma_1=(-1,1)$.
\subsection{The functionals $\Fe$, $\F$, 
$\vert \grad \Fe\vert$, $\vert \grad \F\vert$}\label{sec:ene}
For any $\eps \in (0,1]$ 
we indicate by
$$
\Fe
:\spazio \to [0,+\infty]
$$
the functional defined as
\begin{equation*}
\Fe(
\ustatic
):=
\begin{cases}
  \displaystyle \int_{\toro}\Big( \epss \frac{(\ustatic_x)^2}{2} +W(\ustatic)\Big)\,dx 
         & \text{if $\ustatic_x  \in L^2(\mathbb T)$ 
                 and $W(\ustatic) \in L^1(\toro)$,
 }
  \\
  +\infty & \text{elsewhere},
\end{cases}
\end{equation*}
and by 
$$
\F : \spazio \to [0,+\infty]
$$
the functional defined as
\begin{equation*}
\F(\ustatic):=
\begin{cases}
\displaystyle \int_{\toro} W^{**}(\ustatic)~dx & {\rm if}~ W^{**}(\ustatic) \in L^1(\toro),
\\
+\infty & {\rm elsewhere}.
\end{cases}
\end{equation*}
It is clear that $\F$ is a convex functional.

We denote by 
$$
|\nabla \Fe|
\colon \spazio \to
[0,+\infty]
$$
the functional defined as
\begin{equation*}
|\nabla \Fe|(\ustatic):=
\begin{cases}
 \|\ 
(W'(\ustatic)-\epss  \ustatic_{xx})_{x}
\|_{L^2(\toro)}
                & \text{if $\Fe(\ustatic)<+\infty$ and}
\\ {} &
\text{$(W'(\ustatic)-\epss  \ustatic_{xx})_{x} \in L^2(\toro)$},
  \\
\\
  +\infty & \text{elsewhere},
\end{cases}
\end{equation*}
and by 
$$
\vert\nabla \F\vert: \spazio \to [0,+\infty]
$$
the functional defined as
\begin{equation*}
|\nabla \F|(\ustatic):=
\begin{cases}
 \|(W^{**~\!\prime}(\ustatic))_{x}\|_{L^2(\toro)}
              & \text{if $\F(\ustatic)<+\infty$ 
and $(W^{**~\!\prime}(\ustatic))_x \in L^2(\toro)$},
  \\
  +\infty & \text{elsewhere}.
\end{cases}
\end{equation*}

\section{Statement of the main result}\label{sub:grad}
Given $\eps \in (0,1]$ and 
$\inidatue \in \spazio$ such that
$$
\Fe(\inidatue) < +\infty,
$$
we let 
$\gradflowue
\in \mathcal C^\infty((0,+\infty)\times\toro)\cap \mathcal C^0([0,+\infty);\spazio)$
be the solution 
to the Cauchy problem
\begin{equation}
\label{eq:chgfmarino}
\begin{cases}
{\rm u}_t  = \Big(W'({\rm u})-\epss {\rm u}_{xx}\Big)_{xx}
                      & \text{in~ $(0,+\infty)\times \bb T$},
 \\
{\rm u} ~= \inidatue & \text{on $\{0\}\times \bb  T$}.
\end{cases}
\end{equation}
We notice that $\gradflowue$ is the gradient flow of $\Fe$ in $\spazio$ starting at $\inidatue$ 
in the sense of \cite{AmGiSa:08}, that is, it satisfies:
\begin{itemize}
\item[-] 
$\gradflowue \in 
AC^2
\left([0,+\infty); \spazio\right)$,
where $AC^2\left([0,+\infty); \spazio\right)$
denotes the space of absolutely continuous curves from $[0,+\infty)$ 
to $\spazio$ having derivative in $L^2((0,+\infty))$,
\item[-]
$(0,+\infty) \ni t \mapsto 
\vert \grad \Fe\vert (\gradflowue(t))$ belongs to
$L^2((0,+\infty))$,
\item[-] 
for all $t\geq 0$
\begin{equation}
\label{eq:gf1}
\Fe(\inidatue) = 
\Fe(\gradflowue(t))+ \frac{1}{2} 
       \int_0^t \|\partial_t \gradflowue(s)\|_{-1}^2\,ds  
    + \frac{1}{2} \int_0^t |\nabla \Fe|^2(\gradflowue(s)) \,ds.
\end{equation}
\end{itemize}

A differential characterization of 
the gradient flow of $\F$ in $\spazio$ 
is more delicate, as regularity issues appear. 
Indeed, the function $W^{**}$ is just of class
$\mathcal C^{1,1}(\R)$, and not of class $\mathcal C^2(\R)$. Yet it is 
possible to see that $\vert\nabla \F\vert $ is a strong upper gradient for $\F$ in the sense of \cite[Definition~1.2.1]{AmGiSa:08}, so that from the general theory of maximal monotone operators 
(see for instance \cite[Theorem 3.2]{Br:73}) 
one gets 
the following result.

\begin{proposition}[{\bf Gradient flow of $\F$}]
\label{p:goodgfuf}
Let $\inidatu \in \spazio$ be such that 
$$
\F(\inidatu)<+\infty.
$$
Then there exists a unique gradient flow solution $\gradflowu$ of $\F$ 
starting at $\inidatu$, which satisfies
\begin{itemize}
\item[-]
$\gradflowu
\in AC^2\left([0,+\infty); \spazio\right)$,
\item[-] $(0,+\infty)\ni t \mapsto 
\vert \nabla \F\vert (\gradflowu(t))$ belongs to $L^2((0,+\infty))$,
\item[-] for
all $t\ge 0$
\begin{equation}
\label{eq:gf1chec}
\F(\inidatu)
=
\F(\gradflowu(t))+ \frac{1}{2} 
          \int_0^t \|\partial_t \gradflowu(s)\|_{-1}^2\,ds  
        + \frac{1}{2} \int_0^t |\nabla \F|^2(\gradflowu(s)) \,ds.
\end{equation}
\end{itemize}
\end{proposition}
\noindent Note that $\gradflowu$ solves equation \eqref{eq:hgfbis}
in the sense of distributions. 

We are now in the position to state
the main result of this paper.

\begin{theorem}[{\bf Convergence of solutions}]\label{teo:convergence}
Let $\inidatue,\,\inidatu \in \spazio$ be such that 
$$
\Fe(\inidatue) < +\infty, \qquad
\F(\inidatu) < +\infty.
$$
Suppose that  
\begin{equation}\label{eq:convetempozero}
\lim_{\eps \downarrow 0} \inidatue = \inidatu \qquad {\rm in}~ \spazio
\end{equation}
and 
\begin{equation}\label{limsuptempozero}
\lim_{\eps \downarrow 0}~ \Fe(\inidatue)= \F(\inidatu).
\end{equation}
Then for any $T >0$, 
\begin{equation}\label{eq:convergprima}
\lim_{\eps \downarrow 0} 
\gradflowue
 =\gradflowu
\quad {\rm in} ~ \mathcal C^0([0,T]; \spazio) 
\end{equation}
and 
$$
\lim_{\eps \downarrow 0} \int_0^T \Big(
\vert \nabla \Fe\vert (\gradflowue(t))-
\vert \nabla \F\vert (\gradflowu(t)) \Big)^2 ~dt =0.
$$
In particular 
$$
\lim_{\eps\downarrow 0} \Fe(\gradflowue(t))=\F(\gradflowu(t)), \qquad 
t\ge 0.
$$
\end{theorem}

As already mentioned,  
following \cite{Serfaty}, the main ingredient to prove 
Theorem \ref{teo:convergence} is the following 
(time independent) result, which concerns the 
$\Gamma$-limit
of the slope
in $\spazio$ of the functionals $\Fe$.
\begin{theorem}[{\bf $\Gamma$-liminf of $(\vert \grad \Fe\vert)$}]
\label{p:gammaslope}
Let $\ustatic \in \spazio$ and let
$(\uestatic)$ be a sequence in $\spazio$ such that 
\begin{equation}\label{eq:uu}
\lim_{\eps
  \downarrow 0} \uestatic = \ustatic \qquad {\rm in}~ \spazio
\end{equation}
and
\begin{equation}\label{eq:enequilim}
\sup_{\eps \in (0,1]} \Fe(\uestatic) < +\infty.
\end{equation}
Then
\begin{equation}\label{eq:semigrad}
\liminf_{\eps \downarrow 0} |\nabla \Fe|(\uestatic)
\ge |\nabla \F|(\ustatic).
\end{equation}
\end{theorem}

We expect 
a full $\Gamma$-convergence result
to hold for $(\vert \grad \Fe\vert)$, 
however such result is not needed
in order to prove
Theorem \ref{teo:convergence}.

\section{Proof of Theorem~\ref{p:gammaslope}: preliminary lemmata}
\label{sec:prel}
We first introduce some regularity remarks for fixed $\eps>0$, that will be used in the following to establish uniform estimates.

\begin{Remark}\label{euno}\rm
We have
$$
\Fe(\ustatic) < +\infty \Rightarrow \ustatic \in L^\infty(\toro).
$$
Indeed, for $x_1,x_2\in\toro$,
$$
|\ustatic(x_1) - \ustatic(x_2)| \leq \int_{\toro} \vert
{\ustatic}_x\vert~dx \leq \left(
\int_\toro ({\ustatic}_x)^2~dx
\right)^{1/2} < +\infty.
$$
Hence,
recalling that 
$\displaystyle \int_{\toro} \ustatic ~dx = m$, it follows
$\ustatic \in L^\infty(\mathbb T)$.
\end{Remark}

\begin{definition}[{\bf The function $\chempote(\ustatic)$}]
If $\ustatic$ belongs to the domain of $\vert \grad \Fe\vert$,
we set
$$
\chempote(\ustatic) := W'(\ustatic) - \epss \ustatic_{xx}.
$$
\end{definition}

\begin{Remark}\label{pal}\rm
We have 
$$
\vert \grad \Fe\vert(\ustatic) <+\infty \Rightarrow
\ustatic \in H^3(\toro).$$
 In particular, if $\vert \grad \Fe (v)\vert<+\infty$ then
\begin{equation}\label{punto}
|\nabla \Fe|(\ustatic)=
 \|\ (W'(\ustatic)-\epss  \ustatic_{xx})_{xx}\|_{-1} = 
\sup_{\varphi \in H^1(\toro)}
\left\{ 2 \langle
\chempote(\ustatic)_{xx},\varphi\rangle
- \Vert \varphi_x\Vert^2_{L^2(\toro)}
\right\}.
\end{equation}

Indeed, remembering Remark \ref{euno}, we have $\ustatic
\in L^\infty(\toro)$. Hence, from 
the assumption 
$\Fe(\ustatic) < +\infty$ it follows
\begin{equation}\label{stimadue}
W'(\ustatic)_x = W''(\ustatic) {\ustatic}_x\in L^2(\toro).
\end{equation}
{}From \eqref{stimadue} and 
the assumption $\vert \nabla \Fe\vert (\ustatic) < +\infty$, we obtain
$
\ustatic_{xxx} \in L^2(\toro)
$
and therefore $\ustatic\in H^3(\mathbb T)$.

Such a regularity allows integration by parts in the 
expression obtained of
 $\|\ (W'(\ustatic)-\epss  \ustatic_{xx})_{xx}\|_{-1} $ from the 
rightmost equality in \eqref{e:Hmeno1}, namely \eqref{punto} holds.
\end{Remark}

We next establish uniform bounds to be used for the proof of Theorem \ref{p:gammaslope}.
\begin{lemma}[{\bf Uniform $L^\infty$-bound}]\label{lemmazero}
Let $\ve \in \spazio$ be such that 
\begin{equation}\label{eq:unifenergy}
\sup_{\eps \in (0,1]}
\Big(\Fe(\ve) + \vert
\nabla \Fe\vert(\ve)\Big) < +\infty.
\end{equation}
Then
\begin{equation}\label{stimainfinity}
\sup_{\eps \in (0,1]}
\Vert \ve\Vert_{L^\infty(\toro)} < + \infty.
\end{equation}
Moreover 
$(\ve)$ admits a converging subsequence in $\spazio$.
\end{lemma}

\begin{proof}
{}From Remark \ref{pal} we have $\ve \in H^3(\toro)$
and $\chempote(\ve)  \in H^1(\toro)$. 
Moreover \eqref{eq:unifenergy}
guarantees
\begin{equation}\label{stima}
\sup_{\eps \in (0,1]} \vert \grad \Fe\vert(\uestatic) = 
\sup_{\eps \in (0,1]} 
\Vert {\chempote(\uestatic)}_x
\Vert_{L^2(\toro)}
< +\infty.
\end{equation}
We claim that 
\begin{equation}\label{eq:veve}
\sup_{\eps \in (0,1]}\Vert \chempote(\ve)\Vert_{L^\infty(\toro)}
<+\infty.
\end{equation}
Using assumption \eqref{a1} on $W$ 
and the periodicity 
of $\ve$, it follows
$$
\Big\vert \int_{\toro} \chempote(\ve)~dx
\Big\vert
=
\Big\vert \int_{\toro} W'(\ve)~dx \Big\vert 
\leq 
C \int_{\toro} (1+ W(\ve))~dx,
$$
hence from \eqref{eq:unifenergy}
$$
\sup_{\eps \in (0,1]}\Big\vert \int_{\toro} \chempote(\ve)~dx
\Big\vert<+\infty.
$$
From this estimate and 
\eqref{stima}, claim \eqref{eq:veve} follows.

Let us now show that 
\begin{equation}\label{stimaWprimo}
\sup_{\eps \in (0,1]}\Vert W'(\uestatic)\Vert_{L^\infty(\toro)} < +\infty.
\end{equation}
Since $W'$ is monotone increasing out of a compact set (see Section \ref{subsec:thepot}), 
to show \eqref{stimaWprimo}
it is enough 
to check that 
\begin{equation}\label{stimaWprimomax}
\sup_{\eps \in (0,1]} W'(\uestatic(x_\eps^+)) < +\infty, 
\qquad
\sup_{\eps \in (0,1]}(-W'(\uestatic(x_\eps^-))) < +\infty, 
\end{equation}
where $x_\eps^\pm \in \toro$ are such that 
$$
\uestatic(x_\eps^+) = \max \{\uestatic(x) : x \in \toro\}, \qquad
\uestatic(x_\eps^-) = \min \{\uestatic(x): x \in \toro\}.
$$
We have, using 
${\uestatic}_{xx}(x_\eps^+)\leq 0$ and
${\uestatic}_{xx}(x_\eps^-)\geq 0$,
$$
\Vert \chempote(\uestatic)\Vert_{L^\infty(\toro)} \geq 
 \chempote(\uestatic(x_\eps^+)) \geq W'(\uestatic(x_\eps^+))
$$
and
$$
-\Vert \chempote(\uestatic)\Vert_{L^\infty(\toro)} \leq
 \chempote(\uestatic(x_\eps^-)) \leq W'(\uestatic(x_\eps^-)).
$$
Therefore, thanks to \eqref{eq:veve}, \eqref{stimaWprimomax} is proven,
and 
\eqref{stimainfinity} follows. 

\noindent
The last assertion follows from the compact embedding of $L^\infty(\toro)$ 
in $H^{-1}(\toro)$.
\end{proof}

In the next lemma we introduce a parametrized family $\mu$ of probability 
measures, associated with suitable sequences $(\uestatic)$, the so-called Young measures.
Let $\mc P(\mathbb R)$ be the set of
probability measures on $\mathbb R$. 
For $\lambda \in \mc P(\mathbb R)$ we let ${\rm spt}(\lambda)$ 
be the support of $\lambda$; moreover, 
if $f$ is a continuous function on $\bb R$, we let $\lambda(f) = 
\int_{\R} f~d\lambda$.
If $\lambda :\toro\ni x \mapsto \lambda_x \in \mathcal P(\R)$
is a parametrized family of probability measures, by $\lambda(f)$
we mean the function $\toro\ni x \mapsto \lambda_x(f) \in \R$.

\begin{lemma}[{\bf The measure $\mu$}]\label{l:mu}
Let $\ustatic \in \spazio$ and let
$(\uestatic) \subset \spazio$ be a sequence such that 
\begin{equation}\label{eq:uuuu}
\lim_{\eps
  \downarrow 0} \uestatic = \ustatic \qquad {\rm in}~ \spazio 
\end{equation}
 and satisfying \eqref{eq:unifenergy}.
Then 
there exists  a measurable
map 
$$
\mu: \toro\ni x \mapsto \mu_x \in \mc P(\mathbb R)
$$
for which the following properties hold:
\begin{itemize}
\item[(a)] there exists a constant $M>0$ such that 
$$
{\rm spt}(\mu_x) \subseteq [-M,M] \qquad {\rm for~a.e.}~ x \in \toro;
$$
\item[(b)] $\ustatic=\mu(\imath)$, where
$\imath$ is the identity
map on $\bb R$;
\item[(c)] there exists a subsequence $(\uekstatic)$
such that 
$$
\lim_{k \to +\infty}
\int_{\toro} f(\uekstatic) ~\varphi~dx =\int_{\toro} \mu(f)
~  \varphi~dx, \qquad
f \in \mathcal C^0(\mathbb R), 
~\varphi \in L^1(\toro);
$$
\item[(d)] $\mu(W') \in H^1(\toro)$, and 
$$
\displaystyle \lim_{k \to +\infty}
\chempotek(\uekstatic) 
= \mu(W') \quad {\rm
  weakly~ in~} H^1(\toro) \quad {\rm and~ strongly~ in~ } L^2(\toro).
$$
\end{itemize}
\end{lemma}
\begin{proof}
By Lemma \ref{lemmazero} we have 
\begin{equation}\label{M}
M:= \sup_{\eps \in (0,1]} \Vert \uestatic\Vert_{L^\infty(\toro)} < +\infty.
\end{equation}
Therefore there exists a (not relabeled) subsequence 
such that $\delta_{\uestatic(x)}\otimes dx$ converges
 to $\mu_x \otimes dx$ weakly$^*$ in the space of
measures on $\toro\times
  \bb R$, where $\mu_x \in \mathcal P(\R)$ 
for almost every $x \in \toro$, hence (c) holds for all 
continuous $\varphi$. Being the sequence $(f(\uestatic))$ 
bounded in $L^\infty(\toro)$, the convergence 
holds for any $\varphi \in L^1(\toro)$, and this proves (c).

Since all measures
$\delta_{\uestatic(x)}$ have support in $[-M, M]$ also $\mu_x$ has 
support in $[-M, M]$, which gives (a). Assertion (b) follows 
by taking $f = \imath$ in (c).

{}From Remark 
\ref{pal} and the proof 
of Lemma \ref{lemmazero}, it follows that 
the sequence $(\chempote(\uestatic))$
is bounded in $L^{2}(\toro)$. 
The uniform bound \eqref{eq:unifenergy}
then implies
\begin{equation}\label{dora}
\sup_{\eps \in (0,1]} 
\Vert \chempote(\uestatic)\Vert_{H^1(\toro)} < +\infty. 
\end{equation}
Hence there exists a
(not relabeled) subsequence along which $\chempote(\uestatic)$ 
converge weakly in 
$H^{1}(\toro)$ and strongly in $L^2(\toro)$. On the other hand, $\chempote(\ue)$ converges
to $W'(\ustatic)$ in the sense of distributions on $\toro$. 
By uniqueness of the
  limit, assertion (d) follows.
\end{proof}

The meaning of the next proposition is better illustrated by 
the subsequent Corollary \ref{cor:int} where the assumptions allow, roughly 
speaking, to locally
choose $l = W'$.
\begin{proposition}\label{c:mucorr}
Let $(
\uestatic
)$ and $\mu$ be as in Lemma \ref{l:mu}. 
Let $l \in \mathcal C^0(\bb R)$ be nondecreasing. Then
\begin{equation}\label{eq:mauro}
\mu(l W') \le \mu(l) \mu(W')<+\infty.
\end{equation}
\end{proposition}

\begin{proof}
  Since
$l$ is continuous, 
from Lemma \ref{lemmazero} it follows that
the sequence $(l(\uestatic))$ is bounded in $L^\infty(\toro)$. 
Using Lemma \ref{l:mu} (c),  possibly 
passing to a (not relabeled) subsequence, we have that 
$l(\uestatic)$ converge to $\mu(l)$
  weakly$^*$ in $L^\infty(\toro)$ and strongly in $H^{-1}(\toro)$.
  Then
$$
\begin{aligned}
& \int_{\toro} \vert 
l(\uestatic) \chempote(\uestatic) - \mu(l) \mu(W')
\vert~dx 
\\
\leq &
\int_{\toro} \vert 
(l(\uestatic) - \mu(l))  \chempote(\uestatic) \vert ~dx + 
\int_{\toro} 
\vert \mu(l)(\chempote(\uestatic) -  \mu(W')\vert
~dx 
\\
\leq & 
\Vert l(\uestatic) - \mu(l)\Vert_{H^{-1}(\toro)}
\Vert \chempote(\uestatic)\Vert_{H^1(\toro)}
+ \Vert \mu(l)\Vert_{L^2(\toro)} 
\Vert \chempote(\uestatic) - \mu(W')\Vert_{L^2(\toro)}.
\end{aligned}
$$
Hence, recalling \eqref{dora} and Lemma \ref{l:mu} (d), it follows that
 $l(\uestatic) \chempote(\uestatic)$ converge to $\mu(l) \mu(W')$ 
in $L^1(\toro)$ as $\eps \downarrow 0$.

\noindent 
On the other hand, for all $\varphi \in {\mathcal C}^1(\toro; [0,+\infty))$, 
integrating by parts and using the fact that $l$ is nondecreasing,
\begin{equation}\label{eq:usconi}
\begin{split}
& \int_{\toro} l(\uestatic)~ \chempote(\uestatic)
~\varphi ~dx 
\\ 
= &  \int_{\toro} l(\uestatic)~ W'(\uestatic) ~\varphi ~dx 
         + \epss \int_{\toro} l'(\uestatic)~  ({\uestatic}_x)^2 
~              \varphi ~dx  +  \epss 
 \int_{\toro} l(\uestatic)  ~{\uestatic}_x ~ \varphi_x~dx
\\
\geq &  \int_{\toro} l(\uestatic)~ W'(\uestatic) \varphi \,dx 
 +  \epss
 \int_{\toro} l(\uestatic)  ~{\uestatic}_x~  \varphi_x~dx.
\end{split}
\end{equation}
{}From the uniform bound \eqref{eq:unifenergy} and 
Cauchy-Schwarz's inequality, it follows that the last term
on the right hand side of \eqref{eq:usconi} 
vanishes as $\eps \downarrow 0$.
On the other hand, applying Lemma \ref{l:mu} (c) with the choice
$f = l W'$, we deduce that 
$$
\int_{\toro} 
l(\uestatic) ~ W'(\uestatic)~ \varphi ~dx \to 
\int_{\toro} \mu(l W') ~\varphi~dx.
$$
We conclude 
$$
\int_{\toro} \mu(l) ~\mu(W')~ \varphi ~dx \geq 
\int_{\toro} \mu(l W')~\varphi~dx.
$$
\end{proof}

As a consequence of Proposition \ref{c:mucorr} 
we have the following result which, roughly speaking,
says that the oscillations of a sequence $(\uestatic)$ satisfying 
\eqref{eq:unifenergy}, if contained in 
a connected component of 
$
\R \setminus \lus,
$
namely in an interval where $W'$ is monotone,  are damped down.
This result should be considered together with Lemma \ref{l:youngdelta}
of Section \ref{sec:newprel}, which gives further informations
on $\mu_{x}({W^{**}}')$.

\begin{corollary}[{\bf Support of $\mu_x$, I}]\label{cor:int}
Let $\mu$ be as in Lemma \ref{l:mu}.
For almost every $x \in \toro$ for which
 ${\rm spt}(\mu_x)$ is contained 
in a connected component of $\R \setminus \lus$,
we have that 
$\mu_x$ is a Dirac delta.
\end{corollary}
\begin{proof}
Since the intervals where $W'$ is strictly monotone are 
at most countable, 
we can fix an interval $I$ where $W'$ 
is strictly increasing, and suppose that there exists a set
$A \subseteq \toro$ of positive measure so that for almost every
$x \in A$ the support of $\mu_x$ is contained in $I$.
Choose now a nondecreasing continuous function
$l$ so that $l=W'$ in $I$. 
Then from \eqref{eq:mauro} it follows 
$$
\mu_x(W'^2)\le (\mu_x(W'))^2 \qquad {\rm a.e.}~ x \in A,
$$
which is a reverse Cauchy-Schwarz inequality. It follows that 
$W'$ is constant $\mu_x$-almost everywhere in $A$, and the thesis 
follows recalling that, by assumption, $W$ is not affine in any interval. 
\end{proof}

\section{Localization of oscillations}\label{sec:newprel}
The information gained from the results of the previous section, and in particular 
from Corollary \ref{cor:int},  are not enough to conclude the proof of 
Theorem \ref{p:gammaslope}. 
Our aim now (see Lemma \ref{l:youngdelta}) 
is to prove that for almost
every $x \in \toro$,
either $\mu_x$ is a Dirac delta 
or its support is contained in the closure of a connected component of 
$\gus$. 
The following result, heavily relying on the one-dimensional setting, 
is the crucial step toward the proof of this assertion.

For any $\rho >0$ define 
$$
\gusrho := \{\var \in \R : {\rm dist}(\var, \gus) < \rho\}.
$$
\begin{lemma}[{\bf Localization of oscillations, I}]\label{lemmamatteo}
Let $\uestatic \in \spazio$ and $c \in (0,+\infty)$ be such that 
\begin{equation}\label{eq:unifenergylemma}
\Fe(\uestatic) + \vert
\nabla \Fe\vert(\uestatic) \leq c, \qquad \eps \in (0,1].
\end{equation}
For any $\epseps >0$ there exists
$\delta = \delta(\epseps,c)>0$, depending on $\epseps$ and $c$,
but independent of $\eps$,
such that for any pair $\xe \in \toro$, $\ye \in \toro$
of points satisfying the properties 
\begin{itemize}
\item[(i)] $0 < \ye-\xe \leq \delta$,
\item[(ii)] ${\uestatic}_x(\xe) = 
{\uestatic}_x(\ye) = 0$,
\end{itemize}
we have either 
\begin{equation}\label{ceci}
\uestatic(z) \in  \guse, \qquad z \in [\xe,\ye]
\end{equation}
or 
\begin{equation}\label{secondaposs}
\vert \uestatic(\ye) - \uestatic(\xe)\vert < \epseps.
\end{equation}
\end{lemma}
\begin{Remark}\label{rem:g}\rm
Before proving Lemma \ref{lemmamatteo}, some comments are in order. 
First of all remember that \eqref{eq:unifenergylemma} implies
(Remark \ref{pal}) that $\uestatic
\in H^3(\toro)$, and therefore $\uestatic$ are 
H\"older continuous (in particular uniformly continuous).
This fact, provided we assume $0< \ye-\xe \leq \delta$,
 does not imply inequality \eqref{secondaposs}, since $\epseps$ is 
required not to depend on $\eps$. The second observation concerns the meaning
of Lemma \ref{lemmamatteo}: this lemma states, roughly
speaking,  that between two stationary points
the functions
$\uestatic$
either have
a small oscillation, or they must be close to the set $\gus$
of the $\eps$-independent quantity $\epseps$. In some sense,
if $\uestatic$ have a sufficiently
large excursion between two critical points, their  values
cannot lie inside the region where $W$ is convex. 
Finally, the qualitative behavior of $\delta$ in dependence of $\epseps$
is explicit to a certain extent, see \eqref{contra} below. 
\end{Remark}

\begin{proof}
Fix $\epseps>0$, and let $\xe,\ye \in \toro$ be such that $0 < \ye-\xe$ and 
${\uestatic}_x(\xe) = {\uestatic}_x(\ye)=0$.
For simplicity of notation, in the sequel
of the proof we skip the dependence on $\eps$ of $\xe$
and $\ye$, thus we set $x=\xe$ and $y=\ye$.

Take a point  
$$
z \in [x,y].
$$
We have  
\begin{equation}\label{perdi}
\begin{aligned}
\int_x^z \chempote(\uestatic)~ {\uestatic}_x ~dx= 
&
\int_x^z \left(W'(\uestatic) - {\epss \uestatic}_{xx}\right)~ {\uestatic}_x ~dx
\\
= &  
~ W(\uestatic(z)) - W_\eps(\uestatic(x))
- \frac{\epss}{2} ({\uestatic}_x(z))^2
\\
\leq &
~W(\uestatic(z)) - W_\eps(\uestatic(x)),
\end{aligned}
\end{equation}
and moreover 
\begin{equation}\label{lire}
\int_x^y \chempote(\uestatic) ~{\uestatic}_x ~dx= 
W(\uestatic(y)) - W_\eps(\uestatic(x)).
\end{equation}
On the other hand, 
integrating by parts we have 
$$
\int_x^z \chempote(\ue) ~{\uestatic}_x ~dx =  
- \int_x^z {\chempote(\uestatic)}_x 
~\uestatic~dx + [\chempote(\uestatic) ~\uestatic]^z_x.
$$
Using \eqref{stimainfinity} and 
\eqref{stima}, and recalling assumption \eqref{eq:unifenergylemma},  we have
$$
-\int_x^z {\chempote(\uestatic)}_x ~\uestatic ~dx=  
O\left((z-x)^{1/2}\right),
$$
where $O$ is independent of $\eps$ (while $x$, $y$ and hence
also $z$, depend on $\eps$), so that 
\begin{equation}\label{lireli}
\int_x^z \chempote(\uestatic) ~{\uestatic}_x ~dx= 
[\chempote(\uestatic)~\uestatic]_x^z +
O\left((z-x)^{1/2}\right).
\end{equation}

On the other hand,
using again \eqref{stima}, for the boundary term we have
\begin{equation}\label{incre}
\begin{aligned}    
{} 
[\chempote(\uestatic)~\uestatic]_x^z = &
\chempote(\uestatic(x)) ~[\uestatic]^z_x  + 
\uestatic(z)~[\chempote(\uestatic)]_x^z 
\\
= &
\chempote(\uestatic(x))~ [\uestatic]^z_x  + O\left((z-x)^{1/2}\right),
\end{aligned}
\end{equation}
where $O$ is (another infinitesimal) still independent of $\eps$.
Collecting together \eqref{perdi}, \eqref{lire}, \eqref{lireli} and \eqref{incre} we deduce
\begin{equation}\label{bile}
W(\uestatic(z)) \geq 
W(\uestatic(x))+
\chempote(\uestatic(x)) \big(\uestatic(z)-\uestatic(x)\big) 
+ O\left((z-x)^{1/2}\right), 
\qquad z \in [x,y],
\end{equation}
and at $z=y$,
\begin{equation}\label{uncisipocrede}
W(\uestatic(y)) = 
W(\uestatic(x))+
\chempote(\uestatic(x))\big(
\uestatic(y)-\uestatic(x)\big) + O\left((y-x)^{1/2}\right).
\end{equation}
Assume now 
\begin{equation}\label{assumenow}
\vert \uestatic(y)-\uestatic(x)\vert \ge\epseps.
\end{equation}
Under this assumption we 
can rewrite \eqref{uncisipocrede} 
as 
\begin{equation}\label{faica}
\begin{aligned}
\chempote(\uestatic(x)) = &
s(x,y) + O\Big((y-x)^{1/2}(\uestatic(y)-\uestatic(x))^{-1}\Big)
\\
= &
s(x,y) + O\Big((y-x)^{1/2}/\epseps\Big),
\end{aligned}
\end{equation}
where
$$
s(x,y)
:=\frac{W(\uestatic(y))-W(\uestatic(x))}{\uestatic(y)-\uestatic(x)}\,.
$$
{}From \eqref{bile} and \eqref{faica} we have 
\begin{equation}\label{eqer}
\begin{aligned}
W(\uestatic(z))\ge & ~W(\uestatic(x))
+s(x,y)(\uestatic(z)-\uestatic(x))
\\
&~ +
 O\left((z-x)^{1/2}\right) 
+ O\left((y-x)^{1/2}/\epseps\right)
\\
=& 
~W(\ue(x))
+s(x,y)(\uestatic(z)-\uestatic(x))
+ O\left((y-x)^{1/2}/\epseps\right),
\end{aligned}
\end{equation}
where, again, $O$ is independent of $\eps$. Inequality
\eqref{eqer} says, roughly speaking, that between 
$\uestatic(z)$ and $\uestatic(x)$, the function $W$ must be concave, where
however one must take into account the presence of the error term
$O((y-x)^{1/2}/\epseps)$. For future purposes, it is convenient
to rewrite \eqref{eqer} in the form 
\begin{equation}\label{trag}
W(\uestatic(x))
-W(\uestatic(z))
+s(x,y)(\uestatic(z)-\uestatic(x))
\leq O\left((y-x)^{1/2}/\epseps\right).
\end{equation}
Without loss of generality, in the sequel of the 
proof we assume 
$$
\uestatic(x) \leq \uestatic(y).
$$
Recalling Lemma \ref{lemmazero}, we set 
$$
M := \sup_{\eps \in (0,1]} \|\uestatic\|_{L^\infty(\bb T)} < +\infty.
$$
Given $a, b \in \R$, $a < b$, 
define
$$
\psi(a,b) := 
\max_{c\in [a,b]} \left[W(a)-W(c)
+\frac{W(b)-W(a)}{b-a} (c-a)\right].
$$
Notice that the positivity of $\psi(a,b)$ measures how much 
the function $W$ fails to be concave. Observe also that 
\begin{equation}\label{eq:monti}
\lim_{b \downarrow a} \psi(a,b)=0.
\end{equation}
For any $\rho>0$ let
$\mathcal I_\rho$ be the family  of those intervals
$[a,b] \subset \R$ satisfying the following two properties:
\begin{itemize}
\item[-] $b-a\geq \rho$,
\item[-] $[a,b]$ is not contained in 
$\gusrho$, i.e., 
\begin{equation}\label{eq:gg}
[a,b] \cap
\Big(\R \setminus \gusrho \Big)\neq \emptyset.
\end{equation}
\end{itemize} 
It is convenient to introduce the function 
$\omega: (0,+\infty) \to [0, +\infty]$
defined as follows:
\begin{equation}
\label{eq:sopra}
\omega(\rho):= 
\inf_{[a,b] \subset 
[-M,M], ~
[a,b] \in \mathcal I_\rho
} ~ ~ \psi(a,b). 
\end{equation}
If $\mathcal I_\rho = \emptyset$ (namely, if $\rho>0$ is such that 
there are no
intervals $[a,b]$ contained in $[-M,M]$ with $b-a \geq \rho$ 
and satisfying \eqref{eq:gg} at the same time)
then the infimum on the right hand side of 
\eqref{eq:sopra} is $+\infty$, so that $\omega(\rho) = +\infty$. 
On the other hand, possibly increasing the value of $M$, we can always ensure that 
$\omega<+\infty$ on $(0,\rho_0)$, for some $\rho_0>0$. 
In the sequel we shall assume $\epseps<\rho_0$, so that $\omega(\epseps)<+\infty$.
 
Note that if $\omega(\rho)<+\infty$ then the infimum on the 
right hand side of \eqref{eq:sopra} is a minimum, since $[a,b]$
are constrained to lie in the compact set $[-M,M]$. 
Moreover, recalling that by assumption $W$ is not affine in any interval,
we have
\begin{itemize}
\item[-] $\omega(\rho)>0$,
\item[-] if $\rho_1 < \rho_2$ then $\mathcal I_{\rho_1} \supseteq 
\mathcal I_{\rho_2}$, and therefore 
$\omega$ is nondecreasing;
\item[-]
 $\lim_{\rho\downarrow 0}\omega(\rho)=0$, as a consequence of \eqref{eq:monti}.
\end{itemize}

%
%
Suppose now that 
\begin{equation}\label{isnotcontained}
[\uestatic(x), \uestatic(y)] \rm{~ is~ not~ contained~ in~ \guse}.
\end{equation}
Recalling that $\omega$ is positive, choose $\delta$ be such that 
\begin{equation}\label{contra}
O(\delta^{1/2}/\epseps)\le\frac{\omega(\epseps)}{2},
\end{equation}
where $O$ denotes the remainder term appearing in \eqref{trag}. From \eqref{trag} it then follows
\begin{equation}\label{eqer2}
\begin{aligned}
\max_{z \in [x,y]} \Big(
W(
\uestatic
(x))
-W(\uestatic(z)) +s(x,y)(\uestatic(z)-	\uestatic
(x))\Big)
\leq O(\delta^{1/2}/\epseps) \leq \frac{\omega(\epseps)}{2}.
\end{aligned}
\end{equation}
On the other hand, 
choosing 
$$
a = \uestatic(x), \qquad b =
\uestatic(y)
$$
on the right hand side of \eqref{eq:sopra}, 
and remembering \eqref{assumenow} and \eqref{isnotcontained},
it follows
$$
\max_{z \in [x,y]} \Big(
W(\uestatic(x))
-W(\uestatic(z)) +s(x,y)(\uestatic(z)-\uestatic(x))\Big) \geq \omega(\epseps),
$$
which contradicts
\eqref{eqer2}. We conclude that
\begin{equation}\label{fag}
[\uestatic(x),\uestatic(y)] \subseteq \guse.
\end{equation}
Let us now complete the proof of \eqref{ceci}. 
If $\uestatic(z) \in
[\uestatic(x),\uestatic(y)]$ for any $z \in [x,y]$, from \eqref{fag} 
we deduce $\uestatic(z) \in \Sigma_G^\epseps$, and the proof is concluded.
It remains to consider the case when there exists $z \in (x,y)$ such that 
$$
\uestatic(z) \notin [\uestatic(x),\uestatic(y)]. 
$$
We can assume that 
$\uestatic(z) > \uestatic(y)$, the case $\uestatic(z) < \uestatic(x)$ being similar.
Choose $y'\in [x,y]$ so that $\uestatic(y') = 
\displaystyle 
\max_{\tau \in [x,y]}
\uestatic(\tau)\geq \uestatic(z)$,
and $x' \in [x,y]$ so that $\uestatic
(x') = \displaystyle 
\min_{\tau \in [x,y]}
\uestatic(\tau)\leq \uestatic(z)$. 
Recalling \eqref{assumenow} we have $\vert \uestatic(y') - 
\uestatic
(x')\vert \geq \epseps$.
Therefore we can apply the previous arguments replacing $x$ with
$x'$ and $y$ with $y'$, so that inclusion 
\eqref{fag} reads now as 
$[\uestatic(x'), \uestatic(y')] \subseteq \guse$. This is precisely
inclusion \eqref{ceci}.
\end{proof}

The next lemma says, roughly speaking, that if 
$\uestatic$ asymptotically 
oscillates (as $\eps \downarrow 0$),
then it necessarily does it within the same connected component
of $\gus$. We will focus our attention
on ${W^{**}}'(\uestatic)$, in view of the applications in 
Section \ref{sec:newnewprel}. 
\begin{lemma}[{\bf Support of $\mu_x$, II}]\label{l:youngdelta}
Let $\ustatic$, $(\uestatic)$ and 
$\mu$ be as in Lemma \ref{l:mu}. 
Then, one of the two following alternatives holds:
\begin{itemize}
\item[-] 
for almost every $x \in \toro$ such that 
$\mu_x
({W^{**}}')$ is not contained in 
${W^{**}}'\left(\overline \gus\right)$,
then $\mu_x$ is a Dirac delta;
\item[-] 
for almost every $x \in \toro$ such that 
$\mu_x({W^{**}}')$ is contained in 
${W^{**}}'\left(\overline \gus\right)$,
then 
$\mu_x$ is supported on $\overline{\gus}$.
\end{itemize}
\end{lemma}

\begin{proof}
Define 
$$
\we
:=W^{**~\!\prime}(\uestatic)
$$
which, remembering \eqref{stimainfinity}, is a Lipschitz function
on $\toro$.
We now translate the thesis of Lemma \ref{lemmamatteo} for $w_\eps$.
For $\delta$ as in Lemma \ref{lemmamatteo} 
we set 
$$
\delta'(\epseps):= \delta\left(\frac{\epseps}{2L},c\right)
\qquad \epseps>0,
$$
where $L$ is the Lipschitz 
constant of $W^{**~\!\prime}$ in $[-M,M]$, and $M$ is as in \eqref{M}.
Notice that in the definition of $\delta'$ we need 
$2L$ instead of $L$, to cover the case when \eqref{ceci} holds.

If $x_\eps$ and $y_\eps$
satisfy the assumption of Lemma \ref{lemmamatteo}
with $\delta$ replaced by $\delta'$, we have
\begin{equation}\label{wat}
|\we(x_\eps)-\we(y_\eps)|\le\epseps.
\end{equation}
Observe that this is not a uniform continuity condition on $\we$, since the points
$x_\eps,y_\eps$ are just critical points of $\uestatic$ (and depend on $\eps$), 
and therefore are not
arbitrary points of $\toro$.

Possibly replacing $\delta'$ with its convex envelope,
we can assume
that $\delta'$ is a nonzero convex function 
(tending to zero at zero) in a bounded open interval 
having zero as the left extremum.

{}From Lemma \ref{l:mu} (c) 
we know that 
$$
\lim_{\eps\downarrow 0} 
\we = \mu(W^{**~\!\prime}) =: w \quad {\rm weakly}^*~ {\rm in~} L^\infty(\toro).
$$
We now want to pass from a control on critical points
to a control on the whole of $\toro$. We therefore find
convenient to consider linear interpolations. 
\smallskip

{\it Claim}. Up to extracting a (not relabeled) subsequence, we have 
\begin{equation}\label{caneclaim}
\we\to w \qquad
{\rm almost~ everywhere~in~} \toro ~{\rm as}~ \eps \downarrow 0.
\end{equation}

Let $\hwe \in {\rm Lip}(\toro)$ be such that $\hwe$ is affine in 
each maximal 
open interval of strict monotonicity
of $\uestatic$,
and coincides with $\we$ on the boundary of such an interval. 
Notice that there exists at most a countable number of such intervals.

Let us show that from \eqref{wat} 
it follows
that for all $\epseps>0$ there exists $\delta''(\epseps)>0$ 
independent of $\eps$ such that 
\begin{equation}\label{vadimo}
x\in \toro,~ y \in \toro, \ \ |x-y|\le \delta''(\epseps) \quad  
\Rightarrow \quad
|\hwe(x)-\hwe(y)|\le\epseps.
\end{equation}
To prove \eqref{vadimo} we distinguish two cases. 

First case: 
$x$ and $y$ belong to the same monotonicity interval $I$ 
of $\uestatic$.
Assuming without loss of generality that $x<y$, let $x' \leq x$ and $y' \geq 
y$ be such that $I = (x',y')$. Set 
$$
\lambda := 
\frac{\vert x-y\vert}{\vert x'-y'\vert} \in (0,1].
$$
By construction and 
from \eqref{wat} we know
$$
\vert x' - y'\vert < \delta' \Rightarrow \vert 
\hwe (x') - \hwe(y')\vert<  \epseps.
$$
Hence, as $\hwe$ is affine in $I$, 
\begin{equation}\label{disuggiusta}
\vert x - y\vert < \lambda \delta' \Rightarrow \vert 
\hwe (x) - \hwe(y)\vert<  \lambda \epseps.
\end{equation}
Since $\delta'$ is convex and $\delta'(0)=0$, we have $\lambda \delta'(\epseps)\geq
\delta'(\lambda \epseps)$, and therefore replacing $\epseps$ by $\lambda
\epseps$ and using \eqref{disuggiusta} we deduce \eqref{vadimo} with 
$\delta''$ replaced by $\delta'$. 

Second case:  $x$ and $y$ do not 
belong to the same monotonicity interval  
of $\uestatic$. Assuming without loss of generality that 
$x<y$, 
let $x', y'$ be such that 
\begin{itemize}
\item[-]
$ x\le x' \le y' \le y$,
\item[-]
$ x'$ and $ y'$ are critical points of $\uestatic$,
\item[-] $\hwe$ is strictly
monotone between  $x'$ and $y'$.
\end{itemize}

Then the formula
$$
|x-y| = |x-x'|+|x'-y'|+|y'-y| < \delta'(e) 
$$
implies, using the first case in $[x,x']$ and in $[y',y]$, and
using \eqref{wat} in $[x',y']$, 
$$
|\hwe(x)-\hwe(y)| \le
|\hwe(x)-\hwe(x')|+
|\hwe(x')-\hwe(y')|+|\hwe(y')
-\hwe(y)|
\le 3 \epseps.
$$
That is,
$$
|x-y| < \delta''(\epseps) \Rightarrow 
|\hwe(x)-\hwe(y)| \le \epseps,
$$
where 
$$
\delta''(\epseps) = \delta'(\epseps/3).
$$
This concludes the proof of \eqref{vadimo}.

From \eqref{vadimo} it follows 
that the functions $\hwe$ are equicontinuous, and by Lemma~\ref{lemmazero} they are also  uniformly bounded. We
can apply 
Ascoli-Arzel\`a's
Theorem to get that, possibly passing to a (not relabeled) subsequence,
$$
\hwe\to \widehat w \ {\rm uniformly~ in~ }
\toro, 
$$
for some $\widehat w\in \mathcal C^0(\toro)$.   

For any $n\in\mathbb N$ 
we let $I_1^{\eps,n},\ldots,I_{N_{\eps,n}}^{\eps,n}$ be such that 
$I_j^{\eps,n}=(a_j^{\eps,n},b_j^{\eps,n})\subseteq \toro$ is
a maximal interval
of strict monotonicity of $\uestatic$ and 
\begin{equation}\label{eqdisj}
{\rm osc}(\uestatic;I_j^{\eps,n}):= 
\sup_{I_j^{\eps,n}}\uestatic 
- \inf_{I_j^{\eps,n}}\uestatic \in \left(\frac{1}{n},\frac{1}{n-1}\right],
\qquad j=1,\dots,N_{\eps,n},
\end{equation}
where $N_{\eps,n} \in {\mathbb N} \cup \{+\infty\}$. 
Actually, $N_{\eps,n}$ is finite,
since from Lemma \ref{lemmamatteo} it follows 
$$
|I_j^{\eps,n}|=|a_j^{\eps,n}-b_j^{\eps,n}|>\delta(1/n),
$$ 
so that 
$$
N_{\eps,n} \leq \frac{1}{\delta(1/n)}. 
$$
Up to extracting a further (not relabeled) subsequence, 
we may assume that 
$$
N_{\eps,n} =  N_n,
$$
where $N_n$ depends only on $n$,
and
$$
a_j^{\eps,n}\to a_j, \qquad b_j^{\eps,n} \to b_j^n \quad {\rm  as}~ 
\eps\downarrow 0, 
\qquad
j\in \{1,\ldots,N_n\}.
$$

Let $I^{\eps,n}:=\cup_{j=1}^{N_{n}}I_j^{\eps,n}$ 
and $I^n:=\cup_{j=1}^{N_{n}}I_j^n$. 
Notice that from \eqref{eqdisj} it follows 
$$
I^n\cap I^m=\emptyset \quad {\rm if}~ n\ne m.
$$
For any interval $[a,b]\subset \cup_{n\in\mathbb N}I^n$, the functions $\we$ are monotone on $[a,b]$ for all $\eps>0$ small enough.
As a consequence, 
up to a further subsequence,
\begin{equation}\label{eq:insalata}
\we \to w \qquad {\rm a.e.~on~} 
\cup_{n\in\mathbb N}I^n \ \ {\rm as}~ \eps \downarrow 0.
\end{equation}
On the other hand, given $n \in \mathbb N$ and $x\in 
\toro\setminus (\cup_{m\in\mathbb N}\overline{I^m})$, we have 
${\rm dist}(x,I^\eps_n)\ge c(n)>0$ for all $\eps>0$ small enough, so that
\begin{equation}\label{quella}
|\we(x)-\widehat w(x)|\le 
|\we(x)-\hwe(x)| + |\hwe(x) -\widehat w(x)|\le 
\frac{1}{n}+
\frac{1}{n},
\end{equation}
for $\eps>0$ small enough.
By the arbitrariness of $n\in\mathbb N$ we then get 
$\we\to\widehat w$ uniformly on 
$\toro \setminus 
\cup_{m\in \mathbb N} \overline{I^m}$ as $\eps\downarrow 0$.
This shows that 
\begin{equation}\label{eq:www}
\widehat w = w \qquad {\rm in~} 
\toro \setminus 
\bigcup_{m\in \mathbb N} \overline{I^m}.
\end{equation}
Then \eqref{eq:insalata} and \eqref{eq:www} 
conclude the proof of claim \eqref{caneclaim}. 
\smallskip

Eventually, we show that the claim implies the thesis of the lemma. 
Indeed, for 
almost every $x\in\toro$ such that $w(x)\not\in W^{**~\!\prime}(
\overline \gus)$, 
by the strict monotonicity of $W^{**~\!\prime}$ we have 
$$
\uestatic(x)\to \ustatic(x) \quad {\rm as}~\eps\downarrow 0,
$$
which implies $\mu_x=\delta_{\ustatic(x)}$. 
On the other hand, for 
almost every $x\in\toro$ such that $w(x)\in W^{**~\!\prime}(
\overline \gus)$, 
 we have ${\rm dist}(\uestatic(x),\gus)\to 0$ as $\eps \downarrow 0$,
which implies  
${\rm spt}(\mu_x)\subseteq\ovgus$.
\end{proof}

\nada{
\begin{corollary}
Let $u$ be as in Lemma \ref{l:mu}. Then 
${W^{**}}'(u) \in H^1(\toro)$.
\end{corollary} 

\begin{proof}
Let $\mu$ be as in Lemma \ref{l:mu} (d), so that
$$
\mu(W') \in H^1(\toro).
$$
In particular, $\mu(W')$ is $\alpha$-H\"older continuous in $\toro$,
for some $\alpha \in (0,1)$.

Define 
$$
\Omega := \{x \in \toro: \mu_x(W') \notin {W^{**}}'(\ovgus)\}.
$$
Then $\Omega$ is an open set, hence a countable
union of disjoint open intervals. Moreover, for almost
every  $x \in \Omega$ we have, 
being $W(u(x)) = W^{**}(u(x))$, 
$$
\mu_x(W') = 
W'(u(x)) =  {W^{**}}'(u(x)).
$$
Therefore
$$
{W^{**}}'^{-1}(W'(u(x))) = u(x), \qquad {\rm a.e.}~ x \in \Omega.
$$
In particular $u$ has a uniformly continuous 
representative in $\Omega$. 

We now observe that if $x \to \partial \Omega$,
then ${\rm dist}(u(x), \ovgus)\to 0$.

Using Lemma \ref{l:youngdelta}, it follows that for almost every $x \in \toro$ we have
$$
u(x)\in {
W^{**}}'^{-1}(\mu_x(W')).
$$
Set 
${W^{**}}'(\overline \Sigma_i)= \alpha_i \in \R$

Define the open set 
$\Omega := \{y\in \R: y \notin {W^{**}}'(\ovgus)\}$.
Then $\Omega = \{u \notin\ovgus\}$,
$u= W**'^{-1}(w)\in C^0(\Omega)$ e quindi $W^{**}'(u)=w$ in $\Omega$.
Moreover $u$ e' estendibile per continuita' su $\overline\Omega$.

Ne segue $W^{**}'(u)
\in H^1$ e $\int_T |W**'(u)_x|^2 = \int_\Om |\mu(W')_x|^2\le
\int_T |\mu(W')_x|^2$.
\end{proof}
}

A useful consequence of Lemma \ref{l:youngdelta}
is the following. 

\begin{corollary}\label{ause} 
Under the assumptions of Lemma 
\ref{l:youngdelta}, we have 
$$ 
\mu_x({W^{**}}') = 
{W^{**}}'(\ustatic(x)) \qquad {\rm for~a.e.~} x \in \toro. 
$$ 
\end{corollary} 
\begin{proof} 
If $\mu_x ({W^{**}}')$ is not contained in 
${W^{**}}'\left(\overline \gus\right)$, then $\mu_x$ is a Dirac delta, 
and the assertion follows. 
If $\mu_x({W^{**}}')$ is contained in ${W^{**}}'\left(\overline 
\gus\right)$, then $\mu_x$ is supported on $\overline{\gus}$,
where ${W^{**}}'$ is constant.
\end{proof}

We now improve Lemma \ref{lemmamatteo}
and deduce two corollaries, 
which will be necessary in the proof of Theorem \ref{p:gammaslope}. 
For clarity of exposition, we prefer to state the next lemma
separately from Lemma \ref{lemmamatteo}, even if its proof remains almost
unchanged.

\begin{lemma}[{\bf Localization of oscillations, II}]\label{lemmamatteoimprove}
Let $(\uestatic)
 \subset \spazio$ be a sequence of functions satisfying the bound 
\eqref{eq:unifenergylemma}.
For any $\epseps >0$ and  $C>0$
\begin{itemize}
\item[-] there exists
$\delta = \delta(\epseps,c)>0$, 
depending on $\epseps$ and $c$, but independent of $\eps$ and
$C$, 
\item[-]  there exists $\eps_0 = \eps_0(\epseps,c,C)>0$ 
depending on $\epseps$, $c$  and $C$,
\end{itemize}
such that for any pair $x_\eps\in \toro$,  $y_\eps \in \toro$ of points
satisfying the properties 
\begin{itemize}
\item[(i)] $0 < y_\eps-x_\eps \leq \delta$,
\item[(ii)] 
$\vert {\uestatic}_x(x_\eps)\vert \leq C$,
$\vert {\uestatic}_x(y_\eps) \vert \leq  C$,
\end{itemize}
we have either 
\begin{equation}\label{cecibis}
\uestatic(z) \in  \guse, \qquad z \in [x_\eps,y_\eps], \qquad
\eps \in (0,\eps_0), 
\end{equation}
or 
\begin{equation}\label{secondapossbis}
\vert \uestatic(y_\eps) - \uestatic(x_\eps)\vert < \epseps,
\qquad
\eps \in (0,\eps_0).
\end{equation}
\end{lemma}

\begin{proof}
The proof closely follows the proof of Lemma \ref{lemmamatteo}.
Set $x = x_\eps$ and $y=y_\eps$.
In the present situation, inequality \eqref{perdi} must be replaced by 
\begin{equation}\label{perdibis}
\begin{aligned}
\int_x^z \chempote(\uestatic)~ {\uestatic}_x ~dx
\leq W(\uestatic(z)) - W_\eps(\uestatic(x)) + O(\epss,C),
\end{aligned}
\end{equation}
and equality \eqref{lire} by 
\begin{equation}\label{lirebis}
\int_x^y \chempote(\uestatic)~ {\uestatic}_x ~dx= 
W(\uestatic(y)) - W_\eps(\uestatic(x)) + O(\epss,C),
\end{equation}
where the term $O(\epss,C)$ is actually of the form $O(C^2 \epss)$.
Following the same computations of Lemma \ref{lemmamatteo} 
we must now add on 
the right hand sides of \eqref{lireli}, \eqref{incre}, \eqref{bile},
\eqref{uncisipocrede}, \eqref{faica}, \eqref{eqer}
and \eqref{trag} a remainder term of the form $O(C^2 \epss)$.

Next
we take 
$\eps_0>0$ so that 
\begin{equation}\label{contrabis}
O(C^2 \epss) \leq \frac{\omega(\epseps)}{4}, 
\qquad \eps\in(0, \eps_0),
\end{equation}
and 
 $\delta >0$ so that 
\begin{equation}\label{contrater}
O(\delta^{1/2}/\epseps) \leq \frac{\omega(\epseps)}{4}.
\end{equation}
Then  \eqref{contra} transforms into
$$
O(C^2 \epss)
+ O(\delta^{1/2}/\epseps) \leq \frac{\omega(\epseps)}{2},
$$
and \eqref{eqer2} into
\begin{equation}\label{eqer2bis}
\begin{aligned}
& \max_{z \in [x,y]} \Big(
W(\uestatic(x))
-W(\uestatic(z)) +s(x,y)(\uestatic(z)-\uestatic(x))\Big)
\\
& \leq  O(\delta^{1/2}/\epseps) + O(C^2\epss) \leq \frac{\omega(\epseps)}{2}.
\end{aligned}
\end{equation}
Then the assertions of the lemma follow reasoning along the 
same lines as in the proof of Lemma \ref{lemmamatteo}.
\end{proof}

\begin{corollary}\label{cor:finale}
For any $\epseps>0$ and $c>0$ 
there exist $\eps_0>0$ and $\delta' >0$ 
such that,
if $(\uestatic) \subset \spazio$ is a
sequence of functions satisfying 
\begin{equation}\label{eq:unifenergylemmareprise}
\Fe(\uestatic) + \vert
\nabla \Fe\vert(\uestatic) \leq c, \qquad \eps \in (0,\eps_0),
\end{equation}
and $x \in \toro$ is such that 
$$
{\rm dist}(\uestatic(x),
\gus) \geq 2\epseps, \qquad \eps \in (0,\eps_0),
$$
then 
$$
{\rm dist}\big(\uestatic(y),
\gus\big) \ge \epseps, \qquad y \in (x-\delta', x+\delta'), \ \eps \in 
(0,\eps_0).
$$
%
\end{corollary}

\begin{proof}
By Lemma \ref{lemmazero} there exists $M=M(c)$ such that 
$\sup_{\eps \in (0,1]}\|\uestatic\|_{L^\infty(\toro)}\le M$.
Letting $\delta = \delta(\epseps,c)$ 
be as in Lemma \ref{lemmamatteoimprove}, 
there exist $x_1\in (x-\delta/2,x-\delta/6)$ and $x_2\in (x+\delta/6,x+\delta/2)$ such that 
$|{\uestatic}_x(x_1)|,|{\uestatic}_x(x_1)|\le C:=6M/\delta$. By Lemma \ref{lemmamatteoimprove} there exists $\eps_0$ such that,
if $\eps \in (0,\eps_0)$, then  $|\uestatic(x_1)-
\uestatic(x_2)|<\epseps$. 
We now claim that 
\begin{equation}\label{eqdist}
{\rm dist}(\uestatic(y),
\gus)\ge \epseps \qquad  {\rm for\ all\ }y\in [x_1,x_2], 
\end{equation}
which implies the thesis 
since $(x-\delta', x+\delta')\subset (x_1,x_2)$, with $\delta'=\delta/6$. 
Indeed, letting $y_1$ (resp. $y_2$) be a minimum point (resp. a maximum point) of $\uestatic$ on $[x_1,x_2]$,
again by Lemma \ref{lemmamatteoimprove} we have 
$|\uestatic(y_1)-\uestatic(y_2)|<\epseps$ so that 
\[
|\uestatic(y)-\uestatic(x)|\le |\uestatic(y_1)-\uestatic(y_2)|<\epseps
\]
for all $y\in [x_1,x_2]$, which gives \eqref{eqdist}.
\end{proof}

In general we cannot expect the limit function $\ustatic$ 
to be continuous. Nevertheless, we can prove the following 
results. Recall the definition of $\Sigma_1,\dots,\Sigma_\ell$
given in Section \ref{subsec:thepot}.

\begin{corollary}\label{cor:berlus}
Let $(\uestatic) \subset \spazio$ be a sequence satisfying
the uniform bound \eqref{eq:unifenergy}
and let $\ustatic \in \spazio$ be such that
\begin{equation}\label{eq:uuuubis}
\lim_{\eps
  \downarrow 0} \uestatic = \ustatic \qquad {\rm in}~ \spazio.
\end{equation}
Then the set 
$$
\Omega:=\{x \in \toro: \ustatic(x) \notin \gus\}
$$
has an open Lebesgue representative, and 
\begin{equation}\label{eqlim}
{\rm ess}\!\!\!\!\!\!\!\!\!\lim_{\Omega\ni x\to \overline x\in\partial\Omega} 
~{\rm dist}(\ustatic(x),\gus)=0 .
\end{equation}
Moreover, the sets
$$
C_i := \Big\{x \in \toro : \ustatic(x) \in \overline \Sigma_i\Big\}, 
\qquad i=1,\dots,\ell,
$$
have closed Lebesgue representatives and 
\begin{equation}\label{distpos}
{\rm dist}(C_i, C_j) >0, \qquad i,j = 1,\dots, \ell, \  i \neq j.
\end{equation}
\end{corollary}

\begin{proof}
Let $x\in\Omega$ be a Lebesgue point of $\staticu$ such that ${\rm dist}(\staticu(x), 
\gus)\ge 3\epseps>0$. Letting $\delta'>0$ be as
in Corollary \ref{cor:finale}, for all $\eps>0$ small enough there exists $x_\eps\in (x-\delta'/2,x+\delta'/2)$
such that $|\uestatic(x_\eps)-\staticu(x)|<\epseps$, 
so that ${\rm dist}(\uestatic(x_\eps), 
\gus)\ge 2\epseps$.
By Corollary \ref{cor:finale} it follows ${\rm dist}(\uestatic(y), 
\gus)\ge \epseps$ for all 
$y\in (x_\eps-\delta',x_\eps+\delta')\supset(x-\delta'/2,x+\delta'/2)$, which in turn implies
$${\rm dist}(\staticu(y), \gus)\ge \epseps \qquad {\rm for\ all\ }y\in (x-\delta'/2,x+\delta'/2).
$$
It follows that 
\[
(x-\delta'/2,x+\delta'/2)\subset\Omega
\]
and \eqref{eqlim} holds.
The assertion concerning the sets $C_i$ can be proved similarly. 
Indeed, 
since $\cup_{i=1}^\ell C_i=\toro\setminus \Omega$ has a closed representative,
 it is enough 
to show \eqref{distpos}.
Assume by contradiction there exists $\overline x \in \overline 
C_i \cap \overline C_j$. In this case, in a neighbourhood of $\overline x$ we can find
points $x_\eps$  
such that $\uestatic(x_\eps) \notin \gus$, for $\eps>0$ small enough. Reasoning as 
above, 
this implies $\staticu(x_\eps)\in\Omega$, thus
leading to a contradiction.
\end{proof}

\section{Proof of Theorem~\ref{p:gammaslope}}\label{sec:newnewprel}
We are now in a position to conclude the proof of Theorem \ref{p:gammaslope}. 
Let $\uestatic\to \ustatic$ in $\spazio$ as $\eps \downarrow 0$, 
and choose a subsequence $(\eps_k)\subset (0,1)$ such that 
$$
\lim_{k \to +\infty} \
\vert \grad F_{\eps_k}\vert (\uekstatic) = 
\liminf_{\eps \downarrow 0} \vert \grad F_{\eps}\vert (\uestatic) 
$$
and  
$$
\sup_{k\in \mathbb N} \Big(\Fek(\uekstatic) + 
\vert \grad F_{\eps_k}\vert (\uekstatic)\Big) <+\infty.
$$
Recalling \eqref{punto} we have
\begin{equation}\label{ottokm}
\begin{split}
  & \lim_{k \to +\infty} |\nabla F_{\eps_k}|(\uekstatic) =
\lim_{k \to +\infty}
~  \sup_{\varphi \in H^1(\toro)} \int_{\toro}\Big( 2 
\chempotek
(\uekstatic)_{xx}~
\varphi - (\varphi_x)^2\Big)~dx
\\ \geq & 
\sup_{\varphi \in H^1(\toro)}~\limsup_{k \to +\infty}
 \int_{\toro}\left( 2  \chempotek
(\uekstatic)_{xx}
~\varphi - (\varphi_x)^2\right)~dx.
\end{split}
\end{equation}
Since $(\uekstatic)$ converges to $\ustatic$ in $\spazio$ as $k \to +\infty$, 
we have at our disposal a corresponding measure $\mu$ given by Lemma \ref{l:mu}.  
Using Lemma \ref{l:mu} (d), from \eqref{ottokm} and \eqref{eq:uu} we have
\begin{equation*}
\begin{split}
   \lim_{k \to +\infty} |\nabla F_{\eps_k}|(\uekstatic)
\geq & 
  \sup_{\varphi \in H^1(\toro)} \int_{\toro}\left( 
- 2 (\mu(W'))_x ~ \varphi_x - 
(\varphi_x)^2\right)~dx
\\ 
  =&
\|
(\mu(W'))_x 
\|_{L^2(\toro)}^2
\end{split}
\end{equation*}
(recall from Lemma \ref{l:mu} (d) that $(\mu(W'))_x \in L^2(\toro)$).

\medskip

We now want to show that 
\begin{equation}\label{wenowwa}
\|(\mu(W'))_x \|_{L^2(\toro)}^2 \geq
\|({W^{**}}'(\ustatic))_x \|_{L^2(\toro)}^2.
\end{equation}
Let us define 
$$
\Omega := \Big\{x \in \toro: \ustatic(x) \notin \ovgus\Big\}.
$$
In order to prove \eqref{wenowwa}, 
we will show that ${W^{**}}'(\ustatic)  = \mu (W')$ in $\Omega$,
and that ${W^{**}}'(\ustatic)$ is constant on the connected components 
of $\toro\setminus \Omega$.

By Corollary \ref{cor:finale}, it follows that $\Omega$ has an open 
Lebesgue representative (still denoted by $\Omega$) and 
that, for any $i =1,\dots,\ell$, the set
$$
C_i := \Big\{x \in \toro : \ustatic(x) \in \overline \Sigma_i\Big\}
$$
has a closed Lebesgue representative (still denoted by $C_i$).
Then, 
by Lemma~\ref{l:youngdelta},
$$
\mu_x(W')=W'(\ustatic(x)) = W^{**~\!\prime}(\ustatic(x)) \qquad {\rm for~a.e.~}
x \in \Omega.
$$
Hence, being $\mu_x(W') \in H^1(\toro)$, we get 
$$
{W^{**}}'(\ustatic) 
\in H^1(\Omega).
$$
In particular ${W^{**}}'(\staticu)$ is uniformly continuous on $\Omega$, and 
can be continuously  extended to $\overline \Omega$.
Moreover, for all $\overline x\in\partial\Omega$, from \eqref{eqlim} 
one gets  
that if $x \in \Omega \to \overline x$, then ${\rm dist}(\ustatic(x), \gus)\to 0$,
and 
$$
\lim_{\Omega\ni x \to \overline x,\, x \in \Omega} 
{W^{**}}'(\ustatic(x)) 
\in {W^{**}}'\left(\ovgus\right).
$$
Recalling \eqref{distpos} and the fact that 
${W^{**}}'(\ustatic)$ is locally constant outside $\Omega$, 
it follows 
$$
{W^{**}}'(\ustatic) 
\in H^1(\toro),
$$
and in addition
$$
\Vert ({W^{**}}'(\ustatic))_x\Vert_{L^2(\toro)} = 
\Vert ({W^{**}}'(\ustatic))_x\Vert_{L^2(\Omega)}.
$$
We then have
$$
\begin{aligned}
\lim_{k\to +\infty} |\nabla F_{\eps_k}|(\uekstatic) \ge & 
    \| (\mu(W'))_x \|_{L^2(\bb T)}^2
   \ge \|(\mu(W'))_x \|_{L^2(\Omega)}^2
=
\|(W^{**~\!\prime}(\ustatic))_x \|_{L^2(\Omega)}^2 \\
  = &  \|(W^{**~\!\prime}(u))_x \|_{L^2(\bb T)}^2 =|\nabla \F|(\ustatic).
   \end{aligned}
$$
\qed

\section{Proof of Theorem~\ref{teo:convergence}}\label{sec:obw}
With Theorem \ref{p:gammaslope} at hand, we can prove 
our main convergence result, Theorem \ref{teo:convergence}. 
We will use the standard notation $f(t)(x) = f(t,x)$ for a function
$f \in \mathcal C^0([0,T]; \toro)$.
  
Since $(\Fe(\gradflowue))$ is bounded  by \eqref{eq:gf1} in $[0,T]\times
\toro$, and $W$
has at least linear growth at infinity,
the sequence $(\gradflowue)$ is
uniformly bounded in $L^\infty([0,T]; L^1(\toro))$.
Hence 
$(\gradflowue)$ is bounded in 
$L^\infty([0,T]; \spazio)$ 
and in particular in
$L^2([0,T]; \spazio)$, 
since the subspace of all functions in $L^1(\toro)$ with 
mean $m$ (compactly) embeds in $\spazio$.
Using once more \eqref{eq:gf1} it follows that 
$$
(\gradflowue)~ {\rm is~uniformly~bounded~in~}
H^1([0,T]; \spazio).
$$
Let $(\gradflowuek)$ be a subsequence weakly converging 
in $H^1([0,T]; \spazio)$ to some function $\some$. 
{}From Ascoli-Arzela's theorem in  $H^1([0,T]; \spazio)$,
it follows that $(\uek)$ has a further (not relabelled) subsequence
converging to $w$ in $\mathcal C^0([0,T]; \spazio)$. 
Hence
\begin{equation}\label{segnalato}
\lim_{k\to \infty}\gradflowuek(t) = \some(t) \qquad \forall t\in [0,T]
\end{equation}
\medskip
and 
in particular, recalling \eqref{eq:convetempozero},
\begin{equation}\label{limiteinteso}
\inidatu = \lim_{k\to \infty} \gradflowuek(0) = \some(0).
\end{equation}
\medskip

We now want to show that $w = u$, and to do this
we follow the proof of \cite[Theorem 1]{Serfaty}. 
By assumption \eqref{limsuptempozero}, and remembering
\eqref{eq:gf1}, for any $t \in [0,T]$ we have 
\begin{equation}\label{piano}
\begin{aligned}
{\rm I}:= & \lim_{k\to +\infty}
\left(
\Fek(\gradflowuek(t))+ \frac{1}{2} 
       \int_0^t \|\partial_t \gradflowuek(s)\|_{-1}^2\,ds  
    + \frac{1}{2} \int_0^t |\nabla \Fek|^2(\gradflowuek(s)) \,ds\right)
\\
= & \lim_{k \to \infty} \Fek(\inidatue) =  \F(\inidatu).
\end{aligned}
\end{equation}
On the other hand, 
\begin{equation}\label{piove}
\begin{aligned}
{\rm I}
~ \geq~ & 
\liminf_{k\to +\infty}
\Fek(\gradflowuek(t))
\\
+&  
\liminf_{k \to +\infty}
\frac{1}{2} 
       \int_0^t \|\partial_t \gradflowuek(s)\|_{-1}^2~ds  
    \\
+ &  
\liminf_{k \to +\infty}
\frac{1}{2} \int_0^t |\nabla F_{\eps_k}|^2(\gradflowuek(s)) ~ds.
\end{aligned}
\end{equation}
Applying \eqref{segnalato} and the lower semicontinuity
of $\F$,
it follows
\begin{equation}\label{qui}
\liminf_{k \to +\infty}
\Fek(\gradflowuek(t)) 
\geq
\liminf_{k \to +\infty}
\F(\gradflowuek(t)) 
\geq \F(\some(t)).
\end{equation}
{}From Fatou's Lemma and Theorem \ref{p:gammaslope} 
we have
\begin{equation}\label{nellaforesta}
\liminf_{k \to +\infty}
\int_0^t |\nabla \F|^2(\gradflowuek(s)) ~ds
\geq 
\int_0^t |\nabla \F|^2(\some(s)) ~ds.
\end{equation}
{}From the lower semicontinuity of the norm, and using again 
Fatou's lemma, 
we have 
\begin{equation}\label{nera}
\liminf_{k \to +\infty}
       \int_0^t \|\partial_t \gradflowuek(s)\|_{-1}^2~ds  
\geq 
       \int_0^t \|\partial_t \some(s)\|_{-1}^2~ds.
\end{equation}
Collecting together inequalities 
\eqref{qui}, \eqref{nellaforesta} and \eqref{nera}, 
from \eqref{piove} and \eqref{piano} we infer
\begin{equation}\label{colle}
\F(\inidatu) \geq  
\F(\some(t))+
\frac{1}{2} 
          \int_0^t \|\partial_t \some(s)\|_{-1}^2 ~ds  
        + 
\frac{1}{2} \int_0^t |\nabla \F|^2(\some(s)) ~ ds. 
\end{equation}
On the other hand
we have, using \eqref{limiteinteso}, 
\begin{equation}
\begin{aligned}
& \frac{1}{2} 
          \int_0^t \|\partial_t \some(s)\|_{-1}^2~ds  
        + \frac{1}{2} \int_0^t |\nabla \F|^2(\some(s)) ~ds
\geq - \int_0^t \langle \some_t,  \nabla \F(\some)
\rangle_{\mathcal H^{-1}(\toro)} ~ds
\\
&= - \int_0^t \frac{d}{ds} \F(\some(s))~ds 
= F(\some(0)) - F(\some(t))
= F(\inidatu) - F(\some(t)),
\end{aligned}
\end{equation}
which is the reverse inequality of \eqref{colle}.
Therefore
$$
\F(\inidatu) =
\F(\some(t))+
\frac{1}{2} 
          \int_0^t \|\partial_t \some(s)\|_{-1}^2~ds  
        + 
\frac{1}{2} \int_0^t |\nabla \F|^2(\some(s)) ~ds \qquad 
\forall t\geq 0.
$$
Then $\some$ is the gradient flow of $\F$ 
starting from $\inidatu$, hence
$
\some =\gradflowu.
$
In particular, the whole sequence $(\gradflowue)$ 
converges to $\gradflowu$ and the proof is concluded.
\qed

\appendix\section{}
For completeness, in this appendix
we quickly prove here a $\Gamma$-convergence result 
concerning the functionals $\Fe$. This result
is unnecessary for the proof of Theorem \ref{teo:convergence}.

\begin{proposition}[{\bf $\Gamma$-limit of $\Fe$}]
\label{l:mm}
The sequence $(\Fe)$ $\Gamma$-converges to $\F$ in $\spazio$ as 
 $\eps \downarrow 0$.
\end{proposition}
\begin{proof}

The functional $\F$ is lower semicontinuous in $\spazio$.
Since $\Fe \ge \F$, if $\uestatic\to \ustatic$ in $\spazio$, 
then $\liminf_{\eps\downarrow 0} \Fe(\uestatic)\ge \F(\ustatic)$, 
namely the $\Gamma$-liminf inequality holds.

We now prove the $\Gamma$-limsup inequality: given $\staticu\in\spazio$ 
we have to find a sequence $(\uestatic) \subset \spazio$ with 
\begin{equation}\label{uno}
\uestatic \to \staticu \quad {\rm in}~ \spazio
\end{equation}
 such that
\begin{equation}\label{due}
\lim_{\eps \downarrow 0} \Fe(\uestatic)\to F(\ustatic) \quad{\rm as}~ 
\eps\downarrow 0.
\end{equation}
Assume first that $\ustatic$ is 
piecewise constant and takes values in 
$\R \setminus \gus$. Then, taking a piecewise linear function 
$\uestatic\in H^1(\toro)$ which coincides with $\ustatic$ 
out of a small $\delta_\eps$-neighbourhood of its jump set,  
where $\lim_{\eps \downarrow 0} \frac{\eps}{\delta_\eps}=0$,
and that keeps
the constraint $\int_{\toro} \uestatic ~dx= m$, one gets
\eqref{uno} and \eqref{due}.

It is now enough to show that the class of  functions $\ustatic$ 
considered above
is dense in $\spazio$ and with 
respect to $\F$, so that  the thesis will follow by a standard density argument.
Since 
piecewise constant functions are dense in $\spazio$, 
it is sufficient to show that a piecewise constant function $\ustatic$
can be approximated in $\spazio$ by piecewise constant functions $\ustatic_n$ taking values in $\R\setminus \gus$ and such that
\begin{equation}\label{ene}
\lim_{n\to +\infty} \F(\ustatic_n) =\F(\ustatic).
\end{equation}
Let $\ustatic$ be piecewise constant.
Let $A\subseteq \toro$ be an interval 
where $v$ takes value in $(a,b)$, with $(a,b)$
 a connected component of $\gus$.
Let $\lambda \in (0,1)$ be such that
$\ustatic=\lambda a + (1-\lambda)b$.
We can now take $\ustatic_n \in H^{-1}(A)$ 
such that $\ustatic_n\to \ustatic$ 
in $H^{-1}(A)$,
$\ustatic_n(x)\in\{a, b\}$ for any $x \in A$, and 
$\int_A \ustatic_n~dx = \int_A \ustatic~dx$. Then
$$
\F(\ustatic_n,A):= 
\int_A W^{**}(\ustatic_n)~dx =  
\lambda W^{**}(a) + (1-\lambda)W^{**}(b)
= \F(\ustatic,A) = 
\int_A W^{**}(\ustatic)~dx,
$$
since $W^{**}$ is linear on $[a,b]$. 
We can apply the same argument in the intervals where $\ustatic$
takes values in $\gus$,
while we keep $\ustatic_n=\ustatic$ in the rest of the domain. This concludes the proof.
\end{proof}


\end{document}